\documentclass{amsart}
\usepackage{amsthm,amsmath,amsmath,amssymb,amsbsy,graphicx}

\theoremstyle{plain}
\newtheorem{theorem}{Theorem}

\newtheorem{lemma}[theorem]{Lemma}

\newtheorem{problem}[theorem]{Problem}

\theoremstyle{definition}

\newtheorem{example}[theorem]{Example}

\theoremstyle{remark}

\newcommand{\R}{\mathbb{R}}

\newcommand{\C}{\mathcal{C}}
\newcommand{\REG}{\mathcal{R}}
\renewcommand{\AA}{\mathcal{A}}
\begin{document}

\title{Projection Volumes of Hyperplane Arrangements}

\author{Caroline J. Klivans}
\address{Departments of Mathematics and Computer Science \\ 
The University of Chicago}

\author{Ed Swartz} 
\address{Department of Mathematics \\ 
Cornell University}

\thanks{The work here was done while the first author was a visiting scholar in the mathematics department at Cornell University.  The second author was partially supported by NSF grant DMS-0900912}


\begin{abstract}
We prove that for any finite real hyperplane arrangement the average
projection volumes of the maximal cones is given by the coefficients
of the characteristic polynomial of the arrangement.  This settles the
conjecture of Drton and Klivans that this held for all finite real
reflection arrangements.  The methods used are geometric and
combinatorial.  As a consequence we determine that the angle sums of a
zonotope are given by the characteristic polynomial of the order dual
of the intersection lattice of the arrangement.  \end{abstract} \maketitle

\section{Introduction}
\label{sec:intro}

Given a polyhedral cone $\mathcal{C} \in \mathbb{R}^n$, consider the
orthogonal projection of an arbitrary point $z \in \mathbb{R}^n$ onto
$\mathcal{C}$.  The work here is concerned with the dimension of the
face $z$ projects onto.  Specifically, consider the following problem
as formulated in~\cite{DK}:

\begin{problem}
\label{prob:1}
Which fraction of the unit sphere in $\mathbb{R}^n$, as measured by
surface volume, is occupied by the points $z$ for which the orthogonal  projection of $z$ onto $\mathcal{C}$ lies in the interior of a $k$-dimensional face of $\mathcal{C}$?
\end{problem}

The study of these projection volumes is motivated by $p$-value
calculations in statistical hypothesis testing; see~\cite{DK} for more on this motivation.  
There the projection volumes of fundamental chambers of finite real
reflection arrangements are investigated.  It is shown that for
irreducible reflection groups of type $A_n$, $B_n$ and $D_n$, the
projection volumes are given by the coefficients of the characteristic
polynomial of the corresponding reflection arrangement.  These results
are natural extensions of those of De Conini, Procesi,
Stembridge, and Denham~\cite{DCP,Denham} for zero dimensional projections onto reflection arrangements.

Here we extended and strengthen these results. In particular we
offer a positive answer to the main Conjecture of \cite{DK} which
states that the projection volumes of any finite real reflection
arrangement are given by the coefficients of its characteristic
polynomial.

This follows as a corollary of a stronger result, our
Theorem~\ref{thm:main}, which considers projections onto cones formed
from regions of arbitrary finite real hyperplane arrangements.  In
order to understand this more general case, we shift our perspective
and consider the \emph{average} projection volumes over all regions of
a given hyperplane arrangement.  Theorem~\ref{thm:main}  states that
for an arbitrary finite real hyperplane arrangement the average
projection volumes are given by the coefficients of its
characteristic polynomial.  As all regions of a reflection arrangement are isometric, Corollary $5$ of \cite{DK} holds for all finite real reflection groups.

The methods utilized in \cite{DK} are algebraic, drawing on the
structure of a reflection group.  In contrast,  we use
combinatorial and geometric techniques to prove
our main result.  Projection volumes of cones are related to
the characteristic polynomial through the theory of polytope angle
sums.  To any linear hyperplane arrangement there are naturally
associated dual polytopes, called zonotopes.  A well-known result of
Zaslavsky gives the face numbers of  zonotopes in terms of the
M\"{o}bius function of the intersection lattice of the arrangement.  A
little known result of Perles and Shephard then connects the face
numbers of a zonotope to its angle sums.  Combining
these results, Theorem~\ref{thm:main} equates the angle sums of
certain faces of the zonotopes to both the value of the M\"{o}bius
function at a given intersection and the projection volume onto that
intersection.  We further determine that the angle sums of a zonotope are given by the characteristic polynomial of the order dual of the lattice of flats of the arrangement.

\section{Projection Volumes} \label{projvol}

Let $\C$ be a polyhedral set in $\mathbb{R}^d$. Unless otherwise noted,  $\C$ will always be a cone.   For an
arbitrary point $z \in \mathbb{R}^d$, let $\pi_{\C}(z)$ be
the orthogonal projection of $z$ onto $\C$.  Say that
$\pi_{\C}(z)$ has {\it $k$-dimensional projection} if $\pi_{\C}(z)$
lies in the relative interior of a $k$-dimensional face of
$\C$.  Define $\nu_k$ to be the ratio of volume 
of $\mathbb{R}^d$ occupied by points $x$ for which the projection
$\pi_{\C}$ is $k$-dimensional.  There are several ways of making this precise.  Let $X$ be a cone.  Define the volume of $X$ using any of the following equivalent definitions:

\begin{eqnarray*}
\textrm{Volume of } X
& =& \textrm{The ratio of volume of $\R^d$ occupied by $X$} \\
&=& \displaystyle \frac{|X \cap S^{d-1}|}{|S^{d-1}|} \\
&=& \displaystyle \frac{|X \cap B^d|}{|B^d|},
\end{eqnarray*}

\noindent where $|\ \  |$ is Lebesgue measure, $S^{d-1}$ is the unit sphere, and $B^d$ is the unit ball.

The set of points $z$ which project into the interior of  a face $F$ of $\C$ is a cone which we denote by $X_F$.  Set $\nu(F) =$ the volume of $X_F$. If $F$ is contained in several polyhedral sets, then we use $\nu_\C(F)$ to specify which one.   Finally, we define $\nu_k$ to be   the sum of the $\nu(F)$ over all $k$-dimensional faces $F$ of $\C.$ The $\nu_k$ are called the
projection volumes of  $\C$.  The work here is
motivated by trying to determine the $\nu_k$ for a given cone.  

\begin{example}~\cite[Example 2]{DK}
Consider the cone $\C=[0,\infty)^2$ equal to the non-negative
  orthant in $\R^2$.  All points $z$ in the positive orthant
  $(0,\infty)^2$ lie inside the cone and thus have a $2$-dimensional
  projection $\pi_\C(z)$.  All points in the non-positive
  orthant $(-\infty,0]^2$, the polar cone, are projected to the
origin, that is, they have $0$-dimensional projection.  As all remaining
points have $1$-dimensional projection, the projection volumes are
$\nu_0=\nu_2=1/4$ and $\nu_1=1/2$.  \qed
\end{example}

We will consider polyhedral sets formed by hyperplane arrangements.
We review the basics of the combinatorics of hyperplane arrangements
and refer the reader to~\cite{rstan_hype} for much more.  A real hyperplane
arrangement $\AA$ is a collection of codimension-one affine subspaces
of $\mathbb{R}^d$.  All arrangements appearing in this paper are
assumed finite.  The {\it rank} of an arrangement $\AA$ is defined
to be the dimension of the linear space spanned by the normal vectors
to its hyperplanes.  Specifically, if $\AA = \{H_1, \ldots, H_m\}$
  and $H_i = \{z \in \mathbb{R}^d : \eta_i \cdot z = b_i\}$, where $\eta_i$ is
    a non-zero vector in $\mathbb{R}^d$, then
\[
\textrm{Rk}(\AA) =  \dim(\textrm{Span}\{\eta_1, \ldots,
\eta_m\}).
\]
A region of $\AA$ is any connected component of the complement
of the union of all the hyperplanes in $\AA$.  We use $\REG(\AA)$ to denote  the set of all regions of $\AA.$  When all hyperplanes pass through the origin, and hence the closure of any
region forms a polyhedral cone, the arrangement is called {\it central}.  Generally, the closure of any region
forms a polyhedral cone or polytope.  Note that any polyhedral cone
can be thought of as a region in the hyperplane  arrangement
formed by taking the bounding hyperplanes of the cone.

Much of the combinatorics of a hyperplane arrangement is encoded by
its intersection poset.  Given an arrangement $\AA$, let
$L(\AA)$ be the set of all nonempty intersections of
collections of hyperplanes in $\AA$.  We include
$\mathbb{R}^d$ in $L(\AA)$ as the intersection of the empty
collection.  Define a partial order on $L(\AA)$ by reverse
inclusion of intersections, that is, $x \leq y$ in $L(A)$ if $ y
\subseteq x$.  Then $L(\AA)$ forms a poset ranked by
codimension $d -\dim(x)$.   If the arrangement is central, then  
$L(\AA)$ contains a unique top element $\hat{1}$ and forms a
lattice.  If it is also the case that $\hat{1} = \{0\},$ then the arrangement is called {\it essential.}

 The M\"{o}bius function $\mu$ of a finite poset $P$ is a function from intervals of
 $P$ to $\mathbb{Z}$ defined recursively by:
\begin{align*}
\mu(x,x) & = 1, \textrm{ for all } x \in P\\
\mu(x,y) & = - \sum_{x \leq z < y} \mu(x,z), \textrm{ for all } x < y \in P.
\end{align*} 

Write $\mu(x)$ for $\mu(\hat{0},x)$ when $P$ has a minimal element $\hat{0}.$ 
The characteristic polynomial of a rank $r$ graded poset with $\hat{0}$ and rank function $\rho$ is
\[
\chi_P(t) = \sum_{x \in P} \mu(x) t^{r-\rho(x)}.
\]
If $\AA$ is an essential central arrangement, then for $L(\AA)$ this equals
\[
\chi_{\AA} (t) =  \sum_{x \in L(\AA)} \mu(x) t^{\dim x}.
\]

\begin{example}
\label{ex:3lines-charpoly}
Consider the hyperplane arrangement $\AA \subset\mathbb{R}^2$
consisting of any three lines $H_1$, $H_2$, and $H_3$ through the origin.
 The intersection
lattice of this arrangement is
$L(\AA)=\{\mathbb{R}^2,H_1,H_2,H_3,\{0\}\}$ with its elements ordered
as $\mathbb{R}^2\le H_1\le\{0\}$, $\mathbb{R}^2\le H_2\le \{0\}$, and $\mathbb{R}^2\le H_3\le \{0\}$.  The
M\"obius function thus assigns the values $\mu(\mathbb{R}^2)=1$,
$\mu(H_1)=\mu(H_2)=\mu(H_3) = -1$ and $\mu(\{0\})=2$.  The characteristic polynomial
equals $\chi_\AA (t) = t^2-3t+2$.  \qed
\end{example}

In order to understand projection volumes, we will shift our
perspective and consider the \emph{average} projection volumes over
all regions of a given hyperplane arrangement. 

\begin{example}
For any three lines passing through the origin in $\mathbb{R}^2$, the
average two dimensional volume will always be $\frac{1}{6}$ as there
are $6$ regions.  To determine the zero dimensional volume of a cone,
consider the interior angle $\alpha$ of the cone.  The fraction of
volume which projects onto the vertex of the cone is $\frac{1}{2} -
\alpha$.  Thus the averages volumes are $(\frac{1}{6},\frac{1}{2}, \frac{1}{3})$.
\end{example}

\noindent Our main result, Theorem~\ref{thm:main}, relates average projection volumes to the
 characteristic polynomial.  Before stating the theorem, we
 recall a well-known result due to Zaslavsky~\cite{Zaslavsky} and, independently, Las Vergnas \cite{LV}:
\begin{equation} \label {n(A)}
 (-1)^r\chi_{\AA}(-1) = |\REG(\AA)|,
 \end{equation}
where $|\REG(\AA)|$ denotes the number of regions of $\AA$.
In particular, as the coefficients of $\chi_{\AA}$ alternate in sign \cite[Theorem 4, pg. 357]{Rota}, the sum of the the absolute values of the coefficients of
$\chi_{\AA}(t)$ equals the number of regions of $\AA$.

\begin{theorem}
\label{thm:main}
Let $\AA$ be a rank $r$ real hyperplane arrangement in $\R^d$. Then the sum
$\sum_{\C} \nu_k(\C)$ over all regions $\C$ of
$\mathcal{A}$ is given by the absolute value of the coefficient
of $t^{r-d+k}$ of the characteristic polynomial of $\mathcal{A}$.  
\end{theorem}

\section{Zonotopes}
The link between projection volumes of regions of hyperplane
arrangements and coefficients of characteristic polynomials is via
zonotopes.  Zonotopes are a rich class of polytopes equivalently
defined as affine projections of cubes, Minkowski sums of line
segments, or polytopes with all centrally symmetric faces.
  
  Zonotopes naturally arise from any
essential central hyperplane arrangement.  Throughout this section $\AA$ will always be an essential central arrangement.    Consider the face lattice
$F(\mathcal{A})$ of such an arrangement $\mathcal{A}.$ This poset records the
cellular structure of the decomposition of space as induced by the
hyperplanes.  Since the arrangement is essential and central, $F(\mathcal{A})$ will
have a unique bottom element corresponding to the origin, the atoms
will correspond to one dimensional rays and in general the $j$-dimensional cones are represented by rank $j$ elements of $F(\AA)$.  Next consider a
zonotope $Z$ formed by taking the Minkowski sum of  normals of all the
hyperplanes in $\mathcal{A}$.  Different choices of normals lead to geometrically distinct zonotopes,  but they are all combinatorially equivalent.  Indeed, let $F(Z)$ be the face lattice of $Z$.
Then $F(Z)$ is isomorphic to the order dual of $F(\mathcal{A})$ union
$\hat{0}$; see, for instance, \cite[Chapter 7]{Ziegler}.   Informally, for each region of $\mathcal{A}$ we have a
vertex of $Z$, for each pair of neighboring regions we have an edge of
$Z$, etc.

We will connect projection volumes of regions of $\mathcal{A}$ to
angles associated to the zonotope $Z$.  As an example, consider
$\nu_0(\mathcal{C})$ the zero dimensional projection onto a region
$\mathcal{C}$ of $\mathcal{A}$.  The set of points $x$ such that
$\pi_{\mathcal{C}}(x) = 0$ is given by the normal cone of $\mathcal{C}$,
i.e. the cone generated by all opposites of normals of hyperplanes
supporting $\mathcal{C}$.  The normal cone is simply the translate of
the cone induced by the corresponding vertex figure of $Z$, see
Figure~\ref{fig:normal}.

\begin{figure}[h]
\includegraphics[height=4in]{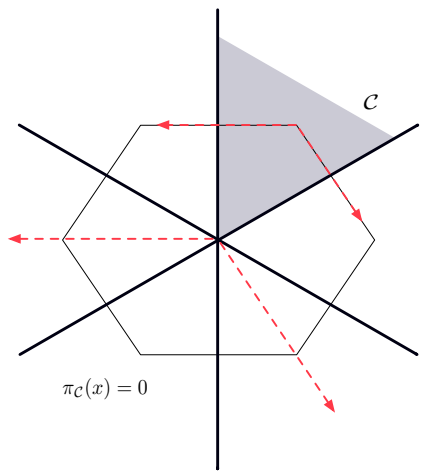}
\caption{An arrangement of three lines with the corresponding zonotope.  The normal cone to $\mathcal{C}$ is given by the dotted rays emanating at the origin.  The cone prescribed by the vertex figure sitting inside $\mathcal{C}$ is a translate of the normal cone.}\label{fig:normal}
\end{figure}

In general, the normal fan of $Z$ formed by taking the normal cones to
all faces of $Z$ is equal to the face fan of $\mathcal{A}$.  As
expected by the duality described above, we have even more generally
that the fan of the arrangement is equal to the face fan of the polar
of the zonotope, see~\cite[Corollary 7.18]{Ziegler}.

\section{Angle Sums of Polytopes}
In order to use  zonotopes to understand $k$-dimensional projections, we will need the theory of angle sums of polytopes.
Let $P$ be a polytope, $F$ a face of $P$ and $z$ a point in the relative
interior of $F$.  Define the angle of $P$ at $F$, $\alpha(P,F)$, as the
ratio of volume of an epsilon ball centered at $z$ which lies inside
$P$.  Specifically, if $B$ is a sufficiently small ball centered at $z,$ then $\alpha(P,F)$ is
the ratio of volume of $B \cap P$ to the volume of $B$.

\begin{example}
Let $F = P$, the improper face of $P$, then $\alpha(P,F) = 1.$
Let $F$ be any facet of $P$, then
$\alpha(P,F) = \frac{1}{2}.$
\end{example}

Define the $k$th angle sum of $P$, $\alpha_k(P)$, as the sum of all angles
over faces of dimension $k$:
$$\alpha_k(P) = \sum_{\textrm{dim}F = k} \alpha(P,F).$$
 
\begin{example}
Let $P$ be a $d$ dimensional polytope.  $\alpha_d(P) = 1$.  
$\alpha_{d-1}(P)  = \frac{1}{2} f_{d-1} = \frac{1}{2} \times \textrm{the number of faces of dimension $d-1$}.$
\end{example}

Angle sums of polytopes satisfy relations similar to those for the
face numbers of a polytope, where the face numbers $f(P) = (f_0, f_1, \ldots, f_{d-1}),$ also called the $f$-{\it vector}, record the number of faces in each dimension; see for example~\cite{Cam}~\cite{PS}.
Furthermore, there are strong connections between the angle sums and the
$f$-vector of a given polytope.  We review here a  result of
Perles and Shephard relating angle sums and face numbers for the class of
equiprojective polytopes.  

Equiprojective polytopes are polytopes such that all projections onto
sufficiently generic hyperplanes have the same $f$-vector.
Formally, let $P_x$ be the polytope obtained by orthogonally projecting
$P$ onto the hyperplane with normal $x$.  A polytope is called
{\it equiprojective} if the face numbers of $P_x$ are the same for all $x$ not parallel to any face of $P$.

\begin{theorem}[Perles-Shephard~\cite{PS}] \label{PS}
Let $P$ be a $d$-dimensional equiprojective polytope and $P^\prime$ a generic projection
of $P$.  Then for $0 \leq k \leq d-1$
$$\alpha_k(P) = \frac{1}{2}(f_k(P) - f_k(P^\prime)),$$
where we set $f_{d-1}(P^\prime) = 0$. 
\end{theorem}

\begin{theorem}[Shephard~\cite{Shephard_zon}] \label{zone}
Zonotopes are equiprojective.
\end{theorem}

In order to use the above theorems we need to understand $f_k(P)$ and
$f_k(P^\prime)$ for a zonotope $P$.  If $P$ is $d$-dimensional,
$f_0(P)$ can be computed from the corresponding arrangement by
(\ref{n(A)}). The general formula is as follows.

\begin{theorem}[Zaslavsky~\cite{Zaslavsky}] \label{f_k}
Let $\AA$ be an  arrangement of hyperplanes in $\R^d.$  Then the face numbers of any zonotope of $\AA$ are given by
$$f_k = \displaystyle\sum_{\dim x = d-k} \displaystyle\sum_{x \le y} (-1)^{\dim x-\dim y} \mu(x,y).$$
\end{theorem}
\noindent In particular, setting $k=0$ recovers (\ref{n(A)}). 

\begin{lemma}
\label{lemma:vertex}
Let $\AA$ be an essential central hyperplane arrangement in $\R^d$ and let $Z$ be an associated zonotope.   Then $\alpha_0(Z) = |\mu_\AA (\hat{0}, \hat{1})|.$
\end{lemma}
\begin{proof}

Write the characteristic polynomial of $\AA$ as
$$\chi_\AA(t) =a_0~t^d - a_1~t^{d-1} + \dots + (-1)^{d-2}a_{d-2}~t^2+ (-1)^{d-1} a_{d-1}~t + (-1)^d~a_d,$$
where all the $a_i$ are nonnegative.  (In fact they are all positive.)
As noted above,  the vertices of $Z$ correspond to the regions of $\AA,$ so equation (\ref{n(A)}) tells us that 
$$f_0(Z) = a_0 + a_1+ \dots + a_{d-2} + a_{d-1} + a_d.$$  
 
 Let $Z^\prime$ be the projection of $Z$ into a generic hyperplane $x$.  What is $f_0(Z^\prime)?$ Evidently, $Z^\prime$ is the Minkowski sum of the images of the projections of the normals $\eta_i$ which determined $Z,$ and hence is a $(d-1)$-dimensional zonotope.   Denote the images of the normals by $\eta^\prime_i$ and the  associated arrangement by $\AA^\prime.$  The intersection poset of any arrangement of linear hyperplanes is completely determined by the  dimensions of the various intersections of the hyperplanes or,  equivalently, by the ranks of all the possible subsets of normal vectors.  The choice of $x$ ensures that for any subset of the $\eta_i$ its rank is the same as the corresponding subset of $\eta^\prime_i$ except when the former has rank $d.$  In that case the rank of the $\eta^\prime_i$ is only $d-1.$ From this we see that $L_{\AA^\prime}$ is just the truncation of $L_\AA.$  Specifically, $L_{\AA^\prime}$  is $L_\AA$ with its coatoms removed.

Now write the characteristic polynomial of $\AA^\prime$ as
$$ \chi_{\AA^\prime}(t) = b_0~ t^{d-1} - b_1~t^{d-2} + \dots + (-1)^{d-3} b_{d-3} ~t^2+ (-1)^{d-2} b_{d-2}~ t + (-1)^{d-1} ~b_{d-1},$$
where all of the $b_i$ are positive.  Since $L_\AA$ and $L_{\AA^\prime}$ agree up to rank $d-2, \ a_i = b_i$ for all $0 \le i \le d-2.$  Theorems \ref{PS}, \ref{zone} and \ref{f_k} imply that $\alpha_0(Z) = \frac{1}{2} (a_{d-1}+a_d - b_{d-1}).$  Since the alternating sums of the $a_i$ and the $b_j$ are zero, $a_{d-1}-a_d=b_{d-1}.$  By definition, $a_d = | \mu_\AA(\hat{0},\hat{1})|.$
\end{proof}

The same argument allows us to compute all of the angle sums of a zonotope in terms of $L_\AA.$

\begin{theorem}
\label{thm:dual}
Let $\AA$ be an essential central hyperplane arrangement and let $Z$ be an associated zonotope.   Then $\alpha_i(Z)$ is the  coefficient of $t^i$ in the characteristic polynomial of the {\bf order dual} of $L_\AA.$
\end{theorem}

\begin{proof}
Applying the same reasoning as in the proof of the above lemma we find that 
$$\alpha_i(Z) = \displaystyle\sum_{\dim x=d-i} \mu_\AA (x, \hat{1}).$$
The result follows as $\mu_\AA(x,\hat{1}) = \mu_{L^*_\AA} (\hat{0}, x)$.  
\end{proof}

\section{Proof of main theorem}

We are now ready to prove our main theorem.

\medskip
\noindent\textbf{Theorem~\ref{thm:main}}.\
\textit{
Let $\AA$ be a central rank $r$ real hyperplane arrangement in $\R^d$. Then the sum $\sum_{\C} \nu_k(\C)$ over all regions $\C$ of $\mathcal{A}$ is given by the absolute value of the coefficient of $t^{r-d+k}$ of the characteristic polynomial of $\mathcal{A}$.  
}
\medskip

\begin{proof}

Our first observation is that it is sufficient to prove this for
essential arrangements.  Indeed, suppose that $\dim \hat{1} = d-r >
0.$ Let $V$ be the orthogonal complement of $\hat{1}$ and consider the
hyperplane arrangement $\AA_V = \{H_1 \cap V, \dots, H_m \cap V\}
\subseteq V.$ Now, $L(\AA_V) \cong L(\AA),$ so their characteristic
polynomials are identical.  Furthermore, the regions of $\AA_V$
correspond bijectively to the regions of $\AA$ by $C \cap V
\leftrightarrow C$ and $\nu_k(C \cap V) = \nu_{k+d-r}(C).$ Thus the
theorem holds for $\AA$ by applying it to the essential arrangement
$\AA_V.$ Hence, from here on we will assume that $\AA$ is an essential
central arrangement.

Let $x \in L(\AA)$ and let $\mathcal{F}_x$ be the set of faces of $\AA$ whose affine span is $x$.  
Our second observation is that it is sufficient to prove 
 
\begin{equation} \label{eq:flats}
\displaystyle\sum_{F \in \mathcal{F}_x} \displaystyle\sum_{\stackrel{F \subseteq \C}{\C \in \REG(\AA)}} \nu_\C(F) = \mu_\AA(x).
\end{equation}

To see that this is sufficient, note that $\nu_k$ is the sum over all pairs $\nu_\C(F)$ with $\dim F =k, \C \in \REG(\AA)$ and $F \subseteq \C.$  Each such face $F$ is contained in exactly one $\mathcal{F}_x$ with the rank of $x$ equal to $d - k,$ and the sum of the $\mu(x)$ of rank $d-k$ is the required coefficient of the characteristic polynomial. 

Fix $Z$  a zonotope for $\AA.$   Recall that an element of the intersection lattice $x \in L(\AA)$ corresponds to a collection of 
parallel faces of $Z$, the duals of the cones $F \in \mathcal{F}_x.$ Each of the faces in the zonotope corresponding to $x$ is formed by taking the sum:
$$\sum \lambda_i \eta_i : -1 \leq \lambda_i \leq 1$$ over all
hyperplanes $H_i = \{x \in \mathbb{R}^d : \eta_i  x = 0 \}$ in $x$ and translating by some $\pm 1$ combinations of the
$\eta_i$ corresponding to hyperplanes not in $x$  \cite[Chapter 7]{Ziegler}.  In particular all of  these faces are isometric.  

Let $Z_x$ be one of the isometric faces of $Z$ corresponding to the intersection $x$ and let $\AA_x$ be the subarrangement of $\AA$ consisting of those hyperplanes which contain $x.$  For any $x<\hat{1},$  
$\AA_x$ is a non-essential arrangement in $\mathbb{R}^d$ but forms an
essential central arrangement if we project onto $V_x$, the
orthogonal complement of $x$.  The lattice of the projected arrangement is isomorphic to  the interval $[\hat{0},x]$ in the original lattice.
Furthermore, we see by the form of $Z_x$ given above, a
zonotope corresponding to the projected arrangement is simply (a translate
of) $Z_x$.  Lemma~\ref{lemma:vertex} then connects the vertex angle sum of $Z_x$ and the M\"{o}bius value at the intersection $x$:
 $$\alpha_0(Z_x) = \mu_{L_{\AA}}(\hat{0},x).$$  

This connects the vertex angle sum of $Z_x$ with the
M\"{o}bius function at $x$. Let $F$ be the cone in $F(\AA)$ whose dual is $Z_x.$  Next we connect the vertex angle sum of
$Z_x$ with $\displaystyle\sum_{\C \in \REG(\AA)} \nu_\C(F).$  

For a particular $\C$, which points of $\R^d$ is $\pi_\C(y)$  in   $F?$
The set of all such points form a
polyhedral cone which we previously denoted $X_F$.  This projection cone is given by the positive span of the normal cone in $V_x$ of the origin with respect to $\C \cap V_x$ and  $F$ itself.   Since these cones lie in orthogonal subspaces they form a product cone.

  Let $w$ be the vertex of (the translated) $Z_x$ which corresponds by
  duality to $\C.$ As the normal fan of $Z_x$ in
  $V_x$ is $\AA_x$, its normal cone is isometric to the cone used to
  determine $\alpha(Z_x, w)$, the angle sum of $Z_x$ at $w$. Thus
  $\nu_\C(F) = \alpha(Z_x,w).$ By duality, for a fixed $F,$ there is a
  one-to-one correspondence between pairs $F \subseteq \C$ and the
  vertices of $Z_x.$ Hence, by Lemma \ref{lemma:vertex}, for a
  particular $F,$
$$\displaystyle\sum_{F \subseteq \C \in \REG(\AA)} \nu_\C(F) = \mu_\AA(x).$$
As the set of all of the interiors $F \in \REG(\AA_x)$ form an open dense subset of $V_x$ of full measure and as previously noted the normal cones are all products, (\ref{eq:flats}) holds and we are done.

\end{proof}

Theorem~\ref{thm:main} can be shown to hold in the greater generality of
all finite affine arrangements. As we know of no applications we only sketch the proof.   The main difficulty in this setting
is extending the notion of projection volume.  The regions of affine arrangements can be arbitrary polyhedral sets and hence  the definition of $\nu(F)$ from Section \ref{projvol} is no longer adequate.  Instead, we again let $X_F$ be the set of all points $z \in \R^d$ such that $\pi_C(z)$ is in the interior of $F.$  But now we define 
$$\nu_\C(F) = \lim_{r \to \infty} \frac{|B_r(a) \cap X_F|}{|B_r(a)|},$$
where $B_r(a)$ is the ball of radius $r$ around a fixed point $a \in \R^d.$ This is
necessary, for example, for bounded faces of the arrangement which are
not zero dimensional.  Such bounded faces in fact have zero projection
volumes. After proving that the above limit exists and is independent of $a$ the argument in the proof of the main theorem can be repeated.

\section{Isometric Regions}

As an obvious corollary to Theorem~\ref{thm:main} we see that for an
arrangement with all isometric regions, the coefficients of the
characteristic polynomial give the \emph{precise} projection volumes
for any fixed region.

Reflection (or Coxeter) arrangements constitute a large class of such
arrangements.  In~\cite{DK}, projection volumes of reflection
arrangements were studied directly heavily utilizing the structure of the reflection group.  The corollary above was
established for certain families of reflection groups and conjectured
for all reflection groups.  Theorem~\ref{thm:main} now affirmatively
answers this conjecture.

We leave as an open question whether or not there exist other examples
of arrangements with isometric cones.

\begin{problem}
Does there exists a real central hyperplane arrangement with all isometric cones that is not a reflection arrangement?
\end{problem}

\end{document}